\documentclass{amsart}
\usepackage{setspace}
\usepackage{a4}
\usepackage{amsthm}
\usepackage{latexsym}
\usepackage{amsfonts}
\usepackage{graphicx}
\usepackage{textcomp}
\usepackage{cite}
\usepackage{enumerate}
\usepackage{amssymb}
\usepackage{hyperref}
\usepackage{amsmath}
\usepackage{tikz}
\usepackage[mathscr]{euscript}
\usepackage{mathtools}
\newtheorem{theorem}{Theorem}[section]

\newtheorem{corollary}[theorem] {Corollary}
\newtheorem{definition}[theorem]{Definition}
\newtheorem{example}[theorem]{Example}

\title{This is the title}
\usepackage{fancyhdr}

\begin{document}
\hrule\hrule\hrule\hrule\hrule
\vspace{0.3cm}	
\begin{center}
{\bf\large{{Finite Field Tarski-Maligranda Inequalities}}}\\
\vspace{0.3cm}
\hrule\hrule\hrule\hrule\hrule
\vspace{0.3cm}
\textbf{K. Mahesh Krishna}\\
School of Mathematics and Natural Sciences\\
Chanakya University Global Campus\\
NH-648, Haraluru Village\\
Devanahalli Taluk, 	Bengaluru  North District\\
Karnataka State, 562 110, India\\
Email: kmaheshak@gmail.com\\

Date: \today
\end{center}

\hrule\hrule
\vspace{0.5cm}
\textbf{Abstract}: Let $\mathbb{F}$ be a sub-modulus field such that $2 \neq 0$. Let $\mathcal{X}$ be a sub-normed linear space over $\mathbb{F}$. Then we show that 
	\begin{align*}
	\bigg|\|x\|-\|y\|\bigg|\leq  \frac{2}{|2|}\|x+y\|+\frac{2}{|2|}\max\{\|x-y\|, \|y-x\|\}-(\|x\|+\|y\|)
\end{align*}
and 
\begin{align*}
	\bigg|\|x\|-\|y\|\bigg|\leq \|x\|+\|y\|-\frac{2}{|2|}\|x+y\|+\frac{2}{|2|}\max\{\|y-x\|, \|x-y\|\}.
\end{align*}
Above inequalities are finite field versions of important Tarski-Maligranda inequalities obained by Maligranda [\textit{Banach J. Math. Anal., 2008}].

\textbf{Keywords}: Normed linear space, Finite field, sub-modulus field, sub-normed linear space.\\
\textbf{Mathematics Subject Classification (2020)}:  12E20, 46B20.\\

\hrule

\hrule
\section{Introduction}

We all know that 
\begin{align}\label{TI}
	\bigg||r|-|s|\bigg|\leq \min\{|r-s|, |r+s|\}, \quad \forall r, s \in \mathbb{R}.
\end{align}
In 1930, Tarski observed  a strengthening of Inequality (\ref{TI}) (Page 688, \cite{TARSKI}):
\begin{align}\label{TE}
		\bigg||r|-|s|\bigg|=|r-s|+ |r+s|-(|r|+|s|), \quad \forall r, s \in \mathbb{R}.
\end{align}
Note that  Equality (\ref{TE}) fails for complex numbers \cite{MALIGRANDA}. In 2008, Maligranda showed that one can generalize Equality (\ref{TE}) by deriving  inequalities that hold for any normed linear space (NLS) \cite{MALIGRANDA}.
\begin{theorem} \cite{MALIGRANDA} \label{MT} (\textbf{Tarski-Maligranda Inequalities})
Let $\mathcal{X}$ be a NLS. Then for all $x, y \in \mathcal{X}\setminus\{0\}$, 
\begin{align}\label{MI1}
	\bigg|\|x\|-\|y\|\bigg|&\leq \|x+y\|+\|x-y\|-(\|x\|+\|y\|)
\end{align}
and 
\begin{align}\label{MI2}
	\bigg|\|x\|-\|y\|\bigg|\leq \|x\|+\|y\|-\bigg|\|x+y\|-\|x-y\|\bigg|.
\end{align}
\end{theorem}
We ask for finite field versions of Inequalities (\ref{MI1}) and (\ref{MI2}). 
We begin by introducing the notion of sub-modulus field. 
\begin{definition}
	Let $\mathbb{F}$ be a  (finite) field. A map $|\cdot|: \mathbb{F}\to [0, \infty)$ is said to be sub-modulus or sub-valued if the following conditions hold.
	\begin{enumerate}[\upshape (i)]
		\item If $\lambda  \in \mathbb{F}$ is such that $|\lambda|=0$, then $\lambda=0$.
		\item $|\lambda \mu|\leq |\lambda||\mu|$ for all $\lambda, \mu \in \mathbb{F}$.
		\item $|\lambda +\mu|\leq |\lambda|+|\mu|$ for all $\lambda, \mu \in \mathbb{F}$.
	\end{enumerate} 
	In this case, we say $\mathbb{F}$ as sub-modulus or sub-valued field. 
\end{definition}
\begin{example}
	Let $p$ be a prime, let $\mathbb{Z}_p\coloneqq \{0, 1, \dots, p-1\}$ be the standard field of integers modulo $p$. Given $n \in \mathbb{Z}_p$, we define 
	\begin{align*}
		|n|\coloneqq \text{unique } a \text{ such that } 0\leq a \leq p-1 \text{ and }n \equiv a (\text{mod }p).
	\end{align*}
\end{example}
We next introduce the notion of sub-normed linear space.
\begin{definition}
	Let $\mathcal{X}$ be a vector space over a sub-modulus field $\mathbb{F}$. A map $\|\cdot\|: \mathcal{X}\to [0, \infty)$ is said to be sub-norm if the following conditions hold.	
	\begin{enumerate}[\upshape (i)]
		\item If $x \in \mathcal{X}$ is such that $\|x\| =0$, then $x=0$.
		\item $\|\lambda x\|\leq |\lambda|\|x\|$ for all $ \lambda \in \mathbb{F}$, for all $x \in \mathcal{X}$.
		\item $\|x+y\|\leq \|x\|+\|y\|$ for all $x,y \in \mathcal{X}$.
	\end{enumerate} 
	In this case, we say $\mathcal{X}$ is a sub-normed linear space. 
\end{definition}
\begin{example}
	Let $p\in [1, \infty)$.  Let $\mathbb{F}$ be a sub-modulus field. For $d \in \mathbb{N}$, let $\mathbb{F}^d$ be the standard vector space. We define 
	\begin{align*}
		\|(a_j)_{j=1}^d\|_p\coloneqq \left(\sum_{j=1}^{d}|a_j|^p\right)^\frac{1}{p}, \quad \forall (a_j)_{j=1}^d \in \mathbb{F}^d.
	\end{align*}
	Then $\|\cdot\|_p$ is a sub-norm on $\mathbb{F}^d$.
\end{example}
\begin{example}
	Let $\mathbb{F}$ be a sub-modulus field. Define 
	\begin{align*}
		c_{00}(\mathbb{N}, \mathbb{F})	\coloneqq\left\{\{a_n\}_{n=1}^\infty: a_n \in \mathbb{F}, \forall n \in \mathbb{N}, a_n \neq 0 \text{ only for finitely many } n\right\}. 
	\end{align*}
	We define 
	\begin{align*}
		\|\{a_n\}_{n=1}^\infty\|_{00}\coloneqq \max_{n\in \mathbb{N}}|a_n|, \quad \forall \{a_n\}_{n=1}^\infty\in c_{00}(\mathbb{N}, \mathbb{F}).
	\end{align*}
	Then $\|\cdot\|_{00}$ is a sub-norm on  $c_{00}(\mathbb{N}, \mathbb{F}).$	
\end{example}

\section{Finite Field  Tarski-Maligranda Inequalities}
It is interesting to note that Inequality (\ref{TI}) may not hold for sub-modulus field. Fundamental reason for this is $|r|\neq |-r|$ in general, in a sub-valued field $\mathbb{F}$. 
For example, in $\mathbb{Z}_3$, $1=|1| \neq |-1|=|2|=2$. Hence 
\begin{align*}
		\bigg||0|-|1|\bigg|=\bigg|-|1|\bigg|=|-1|=|2|=2> \min\{|0-1|, |0+1|\}=\min\{|-1|, |1|\}=\min\{|2|, |1|\}=1.
\end{align*}
We first derive finite field version of Inequality (\ref{TI}). 
\begin{theorem}
Let $\mathcal{X}$ be a sub-normed linear space. Then 
\begin{align*}
		\bigg|\|x\|-\|y\|\bigg|\leq \max\{\|x-y\|, \|y-x\|\}.
\end{align*}
\end{theorem}
\begin{proof}
		Let $x, y \in \mathcal{X}$. Then 
		\begin{align*}
			\|x\|=\|(x-y)+y\|\leq \|x-y\|+\|y\|
		\end{align*}
	and 
		\begin{align*}
		\|y\|=\|(y-x)+x\|\leq \|y-x\|+\|x\|.
	\end{align*}
Hence 
	\begin{align*}
	\|x\|-\|y\|=\|(x-y)+y\|\leq \|x-y\|
\end{align*}
and 
\begin{align*}
	\|y\|-\|x\|=\|(y-x)+x\|\leq \|y-x\|.
\end{align*}
Therefore 
\begin{align*}
		\bigg|\|x\|-\|y\|\bigg|=\max\{\|x\|-\|y\|, \|y\|-\|x\|\}\leq \max\{\|x-y\|, \|y-x\|\}.
\end{align*}
\end{proof}

  Following is finite field version of Theorem \ref{MT}.
  \begin{theorem}
  	(\textbf{Finite Field Tarski-Maligranda Inequalities})
  	Let $\mathcal{X}$ be a sub-normed linear space over sub-modulus field $\mathbb{F}$. Assume $2\neq 0$. Then for all $x, y \in \mathcal{X}$, 
  	\begin{align*}
  		\bigg|\|x\|-\|y\|\bigg|\leq  \frac{2}{|2|}\|x+y\|+\frac{2}{|2|}\max\{\|x-y\|, \|y-x\|\}-(\|x\|+\|y\|)
  	\end{align*}
  	and 
  \begin{align*}
  	\bigg|\|x\|-\|y\|\bigg|\leq \|x\|+\|y\|-\frac{2}{|2|}\|x+y\|+\frac{2}{|2|}\max\{\|y-x\|, \|x-y\|\}.
  \end{align*}
  \end{theorem}
  \begin{proof}
  	Let $x, y \in \mathcal{X}$. Then 
  \begin{align}\label{FF1}
  	\|x\|+\|y\|+\bigg|\|x\|-\|y\|\bigg|=2\max\{\|x\|, \|y\|\}.
  \end{align}	
  We also have 
  \begin{align}\label{A1}
  	\|x+y\|+\|x-y\|\geq \|(x+y)+(x-y)\|=\|2x\|=|2|\|x\|
  \end{align}
  and 
  \begin{align}\label{A2}
  \|x+y\|+\|y-x\|\geq \|(x+y)+(y-x)\|=\|2y\|=|2|\|y\|.
  \end{align}
  Inequalities (\ref{A1}) and (\ref{A2}) give 
  \begin{align}\label{FF2}
  \|x+y\|+\max\{\|x-y\|, \|y-x\|\}=\max\{\|x+y\|+\|x-y\|,\|x+y\|+\|y-x\|\}\geq  |2|\max\{\|x\|, \|y\|\}.
  \end{align}
  Equality (\ref{FF1}) and Inequality (\ref{FF2}) give 
  \begin{align*}
  	\|x\|+\|y\|+\bigg|\|x\|-\|y\|\bigg|=2\max\{\|x\|, \|y\|\}\leq \frac{2}{|2|}\|x+y\|+\frac{2}{|2|}\max\{\|x-y\|, \|y-x\|\}.
  \end{align*}
  Rearranging, 
  \begin{align*}
  	\bigg|\|x\|-\|y\|\bigg|\leq  \frac{2}{|2|}\|x+y\|+\frac{2}{|2|}\max\{\|x-y\|, \|y-x\|\}-(\|x\|+\|y\|).
  \end{align*}
  We next see that 
  \begin{align}\label{E1}
  	\|x\|+\|y\|-\bigg|\|x\|-\|y\|\bigg|=2\min\{\|x\|, \|y\|\}.
  \end{align}	
  We note that 
  \begin{align}\label{F1}
  	\|x+y\|\leq \|(x+y)-(y-x)+\|y-x\|=|2|\|x\|+\|y-x\| \implies 	\|x+y\|-\|y-x\|\leq |2|\|x\|
  \end{align}
and 
 \begin{align}\label{F2}
	\|x+y\|\leq \|(x+y)-(x-y)+\|x-y\|=|2|\|y\|+\|x-y\| \implies 	\|x+y\|-\|x-y\|\leq |2|\|y\|.
\end{align}
Inequalities (\ref{F1}) and (\ref{F2}) give 
\begin{align}\label{F3}
	\|x+y\|-\max\{\|y-x\|, \|x-y\|\}=\min\{\|x+y\|-\|y-x\|, \|x+y\|-\|x-y\|\}\leq |2|\min\{\|x\|, \|y\|\}.
\end{align}
Equality (\ref{E1}) and Inequality (\ref{F3}) give
\begin{align*}
	\|x+y\|-\max\{\|y-x\|, \|x-y\|\}\leq \frac{|2|}{2}(\|x\|+\|y\|)-\frac{|2|}{2}\bigg|\|x\|-\|y\|\bigg|.
\end{align*}
Rearranging, 
\begin{align*}
\bigg|\|x\|-\|y\|\bigg|\leq \|x\|+\|y\|-\frac{2}{|2|}\|x+y\|+\frac{2}{|2|}\max\{\|y-x\|, \|x-y\|\}.
\end{align*}
\end{proof}
  \begin{corollary}
  	Let $p>2$ be a prime number. 	Let $\mathcal{X}$ be a sub-normed linear space over sub-modulus field $\mathbb{Z}_p$. Then for all $x, y \in \mathcal{X}$, 
  	\begin{align*}
  		\bigg|\|x\|-\|y\|\bigg|\leq  \|x+y\|+\max\{\|x-y\|, \|y-x\|\}-(\|x\|+\|y\|)
  	\end{align*}
  	and 
  	\begin{align*}
  		\bigg|\|x\|-\|y\|\bigg|\leq \|x\|+\|y\|-\|x+y\|+\max\{\|y-x\|, \|x-y\|\}.
  	\end{align*}
  \end{corollary}

\section{Conclusions}
\begin{enumerate}
	\item In 1930, Tarski observed an equality in the triangle inequality for real numbers \cite{TARSKI}.
	\item In 2008, Maligranda generalized  Tarski equality to   inequalities which are valid in every normed linear spaces \cite{MALIGRANDA}.
\item In this note, we  derived finite field  version of Tarski-Maligranda inequalities. 
\end{enumerate}

 \bibliographystyle{plain}
 \bibliography{reference.bib}

\end{document}